\documentclass[12pt,twoside]{amsart}

\usepackage{hyperref}
\usepackage{amsthm, amsmath, amscd, amssymb,centernot}
\usepackage{amsfonts}
\usepackage[all]{xy}
\usepackage[T1]{fontenc}
\usepackage[left=2.5cm,top=2.5cm,bottom=3cm,right=2.5cm]{geometry}
\newcommand{\rationals}{\mathbb Q}

\newcommand{\proj}{\mathbb{P}}

\newcommand{\complex}{\mathbb C}

\theoremstyle{plain}
\numberwithin{equation}{section}
\newtheorem{theorem}{Theorem}[section]
\newtheorem*{theorem*}{Theorem}
\newtheorem{proposition}[theorem]{Proposition}
\newtheorem{lemma}[theorem]{Lemma}
\newtheorem{corollary}[theorem]{Corollary}

\theoremstyle{definition}
\newtheorem{question}[theorem]{Question}
\newtheorem{definition}[theorem]{Definition}
\newtheorem{remark}[theorem]{Remark}
\newtheorem{example}[theorem]{Example}

\newtheorem{assumption}[theorem]{Assumption}

\begin{document}
\title{Bounded negativity and Harbourne constants on ruled surfaces}
\author[Krishna Hanumanthu]{Krishna Hanumanthu}
\address{Chennai Mathematical Institute, H1 SIPCOT IT Park, Siruseri, Kelambakkam 603103, India}
\email{krishna@cmi.ac.in}

\author[Aditya Subramaniam]{Aditya Subramaniam}
\address{Chennai Mathematical Institute, H1 SIPCOT IT Park, Siruseri, Kelambakkam 603103, India}
\email{adityas@cmi.ac.in}

\thanks{The first author was partially supported by DST SERB MATRICS
  grant MTR/2017/000243. Both authors were partially supported by a grant from Infosys
  Foundation. }

\date{February 17, 2020}
\maketitle
\begin{abstract}
Let $X$ be a smooth projective surface and let $\mathcal{C}$ be an
arrangement of curves on $X$. The 
Harbourne constant of $\mathcal{C}$ was defined as a way to
investigate the occurrence of curves of negative self-intersection on blow ups of
$X$. This is related to the bounded negativity conjecture which
predicts that the self-intersection number of all reduced curves on a surface is
bounded below by a constant. We consider a 
geometrically ruled surface $X$ over a smooth curve and give lower
bounds for the Harbourne constants of
transversal arrangements of curves on $X$. We also define a global
Harbourne constant as the infimum of Harbourne constants for
arrangements of a specific type and give a lower bound for it. 
\end{abstract}

\section{Introduction}\label{intro}
Let $X$ be a smooth complex projective surface. $X$ is said to have
\textit{bounded negativity} if there exists an integer $b(X)$,
depending only on $X$, such that $C^2 \ge -b(X)$ for all reduced curves $C$ on $X$. The \textit{Bounded Negativity
  Conjecture (BNC)} asserts that every smooth complex projective surface has
bounded negativity. To verify BNC, it suffices to show that
self-intersection of reduced and irreducible
curves is bounded below, by \cite[Proposition 5.1]{B2}. 
While it is easy to prove this conjecture in some
cases (for example, when the anti-canonical divisor $-K_X$ is
effective, it follows from adjunction formula), it is open in
general. For example, the conjecture is open for surfaces obtained by
blowing up at least  ten points on the complex projective plane $\mathbb{P}_{\complex}^2$. 

The notion of \textit{Harbourne constants} was defined in \cite{B1} in an
attempt to understand and clarify the bounded negativity
conjecture. To illustrate the concept, consider the blow up $X$ of
$\proj_{\complex}^2$ at $r$ distinct points. It is clear that the occurrence of negative
curves on $X$ depends on the position of the points that are blown up. For example, if the
points are general enough, it is conjectured that $C^2 \ge -1$ for all
reduced and irreducible curves $C \subset X$. On the other hand, $C^2
= 1-r$ if the points are collinear and $C$ is the strict transform of
the line containing them. The key idea is to divide by $r$ and
consider the ratio $C^2/r$ for all reduced, not necessarily
irreducible, curves $C$ on $X$. The problem then is to bound these
ratios $C^2/r$. The infimum of these ratios as we vary the points on
$\proj^2$ and the reduced curves on blow ups of $\proj^2$ is an
invariant called the \textit{global Harbourne constant}
of $\proj^2$ 
and it is denoted by $H(\proj^2)$. It is not known if
$H(\proj^2) \ne -\infty$. But if $H(\proj^2) \ne -\infty$, then BNC
holds for a blow up of $\proj^2$ at any finite set of points.
One can similarly define the invariant $H(X)$ for any
surface $X$ and if $H(X) \ne -\infty$, then BNC holds for blow ups of
$X$ at finite sets of points; see \cite[Remark 2.3]{B1}.

In order to understand the global Harbourne constant $H(X)$ of a
surface $X$, it is natural to consider the following situation. Let
$\mathcal{C} = \{C_1,\ldots,C_d\}$ be an \textit{arrangement} of
irreducible and reduced curves on $X$. Let $D$ be the effective divisor
$C_1+\ldots+C_d$ on $X$. 
Let $\tilde{X} \to X$ be the blow up of $X$ at the singular points $p_1,
\ldots,p_r$ of $D$
and let $\tilde{D}$ be the strict transform of $D$. We are interested
in the ratio $\frac{\tilde{D}^2}{r}$. As we vary the arrangements
$\mathcal{C}$ on
$X$ and take the infimum of $\frac{\tilde{D}^2}{r}$, we obtain
$H(X)$. So it is natural to first try to bound 
$H(\mathcal{C}) = H(D) : = \frac{\tilde{D}^2}{r}$,
for a specific reduced curve $D$. 

This problem is studied in \cite{B1} when $X = \proj^2$ and all the
irreducible components of $D$ are lines. We say in this case that $\mathcal{C}$
is a \textit{line arrangement}. \cite[Theorem 3.3]{B1} proves that
$H(D) > -4$ for all such $D$. 

Harbourne constants for arrangements of $d$ lines in $\proj^2_k$ for
arbitrary fields $k$ are studied in \cite{DHS}. The \textit{absolute
 linear Harbourne constant} $H(d)$ is defined as the minimum of Harbourne
constants of $d$ lines in $\proj^2_k$ as $k$ varies over all fields. 
The value of $H(d)$ is computed for small values of $d$ and also
special forms of $d$. See \cite[Theorem 1.4, Theorem 1.6]{DHS}.

The case of arrangements of conics on $\proj^2$ was studied in
\cite{PT}. It is proved in \cite[Theorem A]{PT}  that $H(\mathcal{C}) \ge -4.5$ for any such arrangement
$\mathcal{C}$.

The author of \cite{R} considers arrangements $\mathcal{C}$ of
elliptic curves on an abelian surface or on $\proj^2$. 
It is proved
that $H(\mathcal{C}) \ge -4$. Further, in \cite[Theorem 5]{R}, 
a sequence of reduced curves $D_n \subset
\proj^2$ (each of which is a union of elliptic curves) is constructed
such that $\lim_n H(D_n) = -4$. 

In \cite{PRS}, the authors consider reduced divisors $D =
C_1+\ldots+C_d$ on $\proj^2$, where $C_i$ are smooth irreducible
plane curves of degree $n \ge 3$ such that $C_i$ and $C_j$ meet
transversally for all $i \ne j$. Assume also that $d \ge 4$ and that
there are no points in which all the curves meet. Let $s$ be the
number of singular points of $D$. Then they show in \cite[Theorem 4.2]{PRS} 
that $H(\mathcal{C}) \ge -4+ \frac{9nd-5n^2d}{2s} $.

Let $X$ be a smooth hypersurface of degree $d \ge 3$ in $\proj^3$. The
Harbourne constants for line arrangements on $X$ were first studied in
\cite{P}. The bounds obtained there were 
generalized in \cite{LP2}. By \cite[Theorem 3.2]{LP2},
the Harbourne constants of line arrangements $\mathcal{C}$ on $X$ satisfy $H(\mathcal{C}) \ge
-d(d-1)$ when $d \ge 4$. 

Harbourne constants for transversal arrangements of smooth curves on
a surface $X$ with numerically trivial canonical class were studied in \cite{LP1}. The
bounds on Harbourne constants were given in terms of 
the number of curves and the second Chern class of $X$. 
This bound
was generalized to surfaces with non-negative Kodaira dimension in
\cite{LP2}. 

As the above survey of the literature illustrates, most of the
work on Harbourne constants for curve arrangements considered surfaces
of non-negative Kodaira dimension or $\proj^2$. 
In this paper we look at curve arrangements on ruled surfaces and
prove lower bounds on their Harbourne constants. 

The basic tool in studying Harbourne constants for curve arrangements
on surfaces is a method developed by Hirzebruch in \cite{H1}. The idea 
is to consider a branched abelian covering $Z$ of $X$ branched along
the given configuration  $\mathcal{C}$. Then consider the
desingularization $Y$ of $Z$. 
Under some conditions on the surface $X$ and the arrangement
$\mathcal{C}$, $Y$ turns out 
to have non-negative Kodaira dimension. Then one considers 
Hirzebruch-Miyaoka-Sakai type inequalities 
involving the Chern numbers of
$Y$. Hirzebruch described the Chern numbers of $Y$ in terms of certain
invariants of the surface $X$ and certain combinatorial invariants of
the arrangement  $\mathcal{C}$. In the end, one obtains inequalities on
combinatorial invariants of  $\mathcal{C}$ which can then be used to obtain
bounds on Harbourne constants.

Hirzebruch \cite{H1} carried out this procedure for $X = \proj_{\complex}^2$ and for a line
arrangement  $\mathcal{C}$ on $X$ to compute the Chern numbers of
$Y$. In this case, he showed that 
\begin{eqnarray}\label{hirzebruch-type}
t_2 + \frac{3}{4}t_3 \ge d+\sum_{k \ge 5} (k-4)t_k, \text{~if~} t_d=t_{d-1}=0,
\end{eqnarray}
where $d$ is
the number of lines in  $\mathcal{C}$ and $t_i$ is the number of points where
exactly $i$ of the lines in $\mathcal{C}$ meet. Using this inequality crucially, the
authors of \cite{B1} obtain their lower bound on the Harbourne
constant of line arrangements in $\proj^2$ which is mentioned
above. In all the known results on Harbourne constants, a Hirzebruch-type
inequality is used to obtain a bound for the Harbourne
constants. 

An interesting question in this situation is to determine whether the surface $Y$
constructed by the method described above is a
\textit{ball quotient}. These are minimal surfaces of general type whose
universal cover is the 2-dimensional unit ball. Equivalently, they
are minimal surfaces of general type for which 
the Bogomolov-Miyaoka-Yau inequality is an equality. In other words, 
a minimal surface $Y$ is a ball quotient if and only if 
$K_Y$ is nef and big and $K_Y^2 = 3e(Y)$, where $K_Y$ is the canonical
divisor of $Y$ and $e(Y)$ is the topological Euler characteristic of
$Y$. In \cite{H1}, Hirzebruch was interested in constructing ball
quotients by starting with line arrangements on $\proj^2$. 
We show that the surfaces we construct starting with curve
arrangements on ruled surfaces do not give new examples of ball
quotients. We follow the methods developed in \cite{BHH}.


The paper is organized as follows. 

In Section \ref{prelims}, we recall some basic facts of ruled surfaces and
introduce the curve arrangements that we study. We also include some
well-known combinatorial properties of these curve arrangements that
we require. 

In Section \ref{abelian-cover}, using a result of Namba, we construct an abelian
cover $Z \to X$ branched on the given curve arrangement and then consider the 
desingularization $Y \to Z$; see Figure \ref{dia:diagram1}. We also compute the
Chern numbers of $Y$ and relate these to the combinatorial data of the
curve arrangement on $X$. 

In Section \ref{main-results}, 
we first show that $Y$ has non-negative Kodaira dimension which
enables us to apply a Hirzebruch-Miyaoka-Sakai type inequality. Using
this,  we
prove our main results Theorem \ref{Theorem 4.7} and Corollary
\ref{main-corollary} 
about Harbourne constants for curve arrangements
on ruled surfaces. Theorem \ref{Theorem 4.7} gives a lower bound for
Harbourne constants for a specific curve arrangement on a ruled
surface $X$. For a fixed pair of integers $a,b$, we define a \textit{global
Harbourne constant $H_{a,b}(X)$} which is obtained by taking the
infimum of Harbourne constants as the curve arrangements vary (see
Definition \ref{global-harbourne}).  In 
Corollary \ref{main-corollary}, we give lower bound for global
Harbourne constants on any ruled surface. 
Assuming that the curves in our arrangement do not
intersect the normalized section of the ruled surface, we obtain a
better bound for the Harbourne constant in Proposition \ref{Prop
  3.6}. Using these bounds, we give a lower bound in Corollary
\ref{Cor 4.12} for the self-intersection
of the strict transform of the curve arrangement  for the blow up of
all its singular points.

Finally, in Section \ref{ball-quotient}, we show that the surface $Y$ is not a ball
quotient (see Theorem \ref{Theorem 5.2}). 

We work throughout over the complex number field $\complex$.

\section{Preliminaries}\label{prelims} 
\begin{definition}[Transversal arrangement]
Let $\mathcal{C}=\{C_1,C_2, \ldots ,C_d\}$ be an arrangement of curves
on a smooth projective surface $X$.
We say that $\mathcal{C}$ is a \emph{transversal arrangement}
if $d\geq 2$, all  curves $C_i$ are smooth
and they intersect pairwise transversally.
\end{definition}
Given a transversal arrangement $\mathcal{C}=\{C_1,C_2,\ldots, C_d\}$,
we have a divisor $D=\sum\limits_{i=1}^{d}C_i$ on $X$. We use the arrangement
$\mathcal{C}$ and the divisor $D$ interchangeably. 

Let $\text{Sing}(\mathcal{C})$ be the set of all intersection points
of the curves in a transversal arrangement $\mathcal{C}$. Note that $\text{Sing}(\mathcal{C})$ is
precisely the set of singularities of the reduced curve $D$, since all
the irreducible components of $D$ are nonsingular by hypothesis. 
Let $s$ denote the number of points in the set $\text{Sing}(\mathcal{C})$.

\begin{definition}[Combinatorial invariants of transversal arrangements]
Let	$\mathcal{C}$ be a transversal arrangement on a smooth surface
$X$. For a point $p\in X,$
let $r_p$ denote the  number of elements of $\mathcal{C}$ that pass
through $p$. We call $r_p$ the \textit{multiplicity} of $p$ in
$\mathcal{C}.$ We say $p$ is a $k$-fold point of $\mathcal{C}$ 
if there are exactly $k$ curves in $\mathcal{C}$ passing through $p.$
For a positive integer $k\geq2$, $t_k$ denotes the number of $k$-fold points in $\mathcal{C}$.
\end{definition}
 These numbers satisfy the following standard equality, which follows by
counting incidences in a transversal arrangement in two ways:
\begin{equation}\label{eq:combinatorial general}
\sum_{i<j}(C_i\cdot C_j)=\sum_{k\geq 2}\binom{k}{2}t_k.
\end{equation}
Also, let $$f_i=f_i(D) :=\sum_{k\geq 2}k^i t_k.$$
In particular, $f_0=s$ is the number of points in
$\text{Sing}(\mathcal{C})$. 

\begin{definition}[Harbourne constants of a transversal arrangement]\label{def:H-constant TA}
	Let $X$ be a smooth projective surface.
	Let $D=\sum\limits_{i=1}^d C_i$ be a transversal arrangement
        of curves on $X$ with $s = s(D) > 0$.
	The rational number
	\begin{equation*}\label{eq:TA Harbourne constant}
H(X,\mathcal{C})=	H(X,D)=\frac1s\left(D^2-\sum\limits_{P\in \text{Sing}(D)}r_P^2\right)
	\end{equation*}
	is called the \emph{Harbourne constant of the transversal arrangement} $\mathcal{C}$.
\end{definition}

When the surface $X$ is clear from the context, we simply write
$H(\mathcal{C})$ or $H(D)$ to denote the Harbourne constants. 

In this paper, we consider transversal arrangements of curves on ruled surfaces. We follow the
notation in \cite[Chapter V, Section 2]{Har}.   


Let $C$ be a smooth complex curve of genus $g$. A
\textit{geometrically ruled surface} is a
surface of the form $X = \proj_C(E)$ where $E$ is a rank 2 vector
bundle on $C$. We refer to such surfaces simply as \textit{ruled}
surfaces. Let $\phi: X \to C$ be the natural map.

Note that $\proj_C(E) \cong \proj_C(E\otimes \mathcal{L})$ for any line
bundle $\mathcal{L}$ on $C$. Let $E$ be a \textit{normalized} vector
bundle with $X = \proj_C(E)$; this means that 
$H^0(C, E) \neq 0$ and $H^0(C,
E\otimes\mathcal{L} )=0$  for all line bundles $\mathcal{L}$ on $C$
with $\text{deg}(\mathcal{L}) < 0.$
We set $e :=\text{deg}(\wedge^2E)$.  This invariant is uniquely determined by
$X$. 

We fix a
section $C_0$ of $X$ with $\mathcal{L}(C_0)=\mathcal{O}_{\mathbb
  P(E)}(1).$ Let $f$ denote the numerical class of a fiber of $\phi$. 
Then any element of $\text{Num}(X)$ has the form $aC_{0}+bf$ for
$a,b\in\mathbb{Z}$. The intersection product on  $\text{Num}(X)$ is
determined by  $C_0^2 = -e$,\ $C_0\cdot f=1$ and $f^2 =0.$
Any canonical divisor on $X$, denoted by $K_X$, is numerically
equivalent to  $-2C_{0}+(2g-2-e)f.$

Let $X$ be a ruled surface over a smooth complex curve $C$ of genus
$g$ with $e\geq 0.$ 
If an irreducible curve on $X$, different from $C_0$ and $f$, is numerically equivalent to
$aC_{0}+bf$, then $a > 0$ 
and $b\geq ae.$ A divisor on $X$ which is 
numerically equivalent to $aC_{0}+bf $ is ample 
if and only if  $a > 0$ and $b > ae.$

For more details, see \cite[Chapter V, Section 2]{Har}.\\
\begin{assumption}\label{star}
Let $X$ be a ruled surface over a smooth curve of genus $g \ge 0$ with
invariant $e = e(X) \ge
4$. 
Let 
$\mathcal{C}=\{C_1,C_2,\ldots,C_d\}$ be a transversal arrangement of curves on $X$
with $d\geq4$ and
$t_d=0.$  Suppose that all the curves $C_i$ in $\mathcal{C}$ are linearly equivalent to a
fixed divisor $A$ on $X,$ where $A$ is numerically equivalent to
$aC_{0}+bf,$  for $a,b\in\mathbb{Z}$ with $a>0$ and $b\ge ae.$
Note that under these assumptions, $C_i\cdot C_j=2ab-a^2e$ for all
curves $C_i,C_j \in \mathcal{C}$.  
\end{assumption}

\begin{lemma}\label{eq:combinatorial equality1}
Let $\mathcal{C}=\{C_1,C_2,\ldots,C_d\}$ be a transversal arrangement
of curves on a ruled surface $X$ 
satisfying Assumption
\ref{star}. Then we have the following. 
\begin{enumerate}
\item For every curve $C_i \in\mathcal{C},$ we have $\sum_{p\in C_i}(r_{p}-1)=(2ab-a^2e)(d-1).$
\item $f_2-f_1 = \sum_{k\geq2}k(k-1)t_{k} = (2ab-a^2e)d(d-1).$
\end{enumerate}
\end{lemma}
\begin{proof}
First we prove (1). Given a multiple point $p \in C_i, $ $r_{p}-1$ is
the number of curves of the arrangement passing through $p$ different
from $C_i$. As every curve meets 
every other curve in $2ab-a^2e$ distinct points, the expression 
$\sum_{p\in C_i}(r_{p}-1)$ counts all curves of the arrangement
different from $C_i$, $2ab-a^2e$ times each.
So (1) holds. 

The first equality in (2) follows from the definition of $f_2,
f_1$. As $\sum_{C_i \in\mathcal{C}}\sum_{p\in C_i } (r_{p}-1)=\sum_{k\geq2}k(k-1)t_k,$
the second equality  in (2) follows from (1).
\end{proof}

\section{Construction of the abelian cover}\label{abelian-cover}

Our arguments follow the model developed by Hirzebruch in \cite{H1}. These ideas
have been used by several recent authors. 
See \cite{E, P,PRS, PT,R}, for example. 

Let $X$ be a ruled surface over a smooth curve $C$ of genus $g$. Let
$\mathcal{C} = \{C_1,\ldots, C_d\}$ be a transversal arrangement of curves on $X$ satisfying Assumption
\ref{star}. Our goal is to give bounds for the Harbourne constant
$H(X,\mathcal{C})$. 
The starting point is to consider a branched covering of $X$ branched
along the curves in $\mathcal{C}$. 
In order to prove that such a branched covering does in fact exist
 for the ruled surface $X$, 
we use a result of Namba, which we recall below.

 As above, let $D=\sum_{i=1}^d C_{i}$. Let $\text{Div}(X,D)$ be the subgroup
 of the $\mathbb{Q}$-divisors on $X$ generated by all the integral
 divisors and the following $\rationals$-divisors: $\frac{C_1}{2},
 \frac{C_2}{2}, \ldots, \frac{C_d}{2}$.
 

 Let $\sim$ be linear equivalence in $\text{Div}(X,D)$, where $G\sim G'$
 if and only if $G-G'$ is an integral principal divisor. Let 
$\text{Div}^{0}(X,D)/\sim$ denote the kernel of the
first Chern class map:
 \[
 \begin{array}{ccc}
 \text{Div}(X,D)/\sim & \rightarrow & H^{1,1}(X,\mathbb{R})\\
 G & \mapsto & c_{1}(G)
 \end{array}
 \]
 
We use the following result of Namba \cite[Theorem 2.3.20]{N}. In our
special case, it says the following. 

 \begin{theorem}[Namba]
 	\label{thm:(Namba).-There-exists}
 	There exists a finite abelian cover $Z \to X$ with branch
        locus equal to $D$ and ramification index 2 at each $C_i$ 
        if and only if for every
 	$j=1,\dots,d$, there exists an element of finite order $v_{j}=\sum\frac{a_{ij}}{2}C_{i}+E_{j}$
 	of $Div^{0}(X,D)/\sim$, where $E_{j}$ are integral divisors
        and $a_{jj}\in \mathbb{Z}$ is odd for every $j = 1,\dots,d$. 

 	In this case, the subgroup of $Div^{0}(X,D)/\sim$ generated by the $v_{j}$
 	is isomorphic to the Galois group of the abelian cover $Z \to X$.
 \end{theorem}

Set $v_1=v_2=\frac{C_1-C_2}{2}$ and $v_{j}=\frac{C_1-C_{j}}{2}$ for
$j=3,\dots,d$ and $E_j=0$ for every $j$. Then, by Theorem \ref{thm:(Namba).-There-exists}, 
there exists an
abelian cover $\pi:Z\to X$ ramified over $\mathcal{C}$ with
ramification index $2$. The Galois group $G$ of $\pi$ is
generated by $v_1=v_2, v_3,\ldots,v_{d}$ and no proper subset of
$\{v_2,\ldots,v_d\}$ generates $G$. Note that 
every element of $G$ has order 2. So 
the Galois group of $\pi$ is
$(\mathbb{Z}/2\mathbb{Z})^{d-1}$. 
We denote by $\rho:Y\to Z$ the minimal desingularization of $Z$.

For a singular point $p$ of $\mathcal{C},$ recall that $r_p$ denotes
its multiplicity. Let $\tau: \widetilde{X} \to X$ be the blow up of $X$ at the $f_{0}-t_{2}=\sum_{k\geq3}t_{k}$
singular points of $\mathcal{C}$ with multiplicities $k\geq3$. Let
$\widetilde{D}=\sum_{i=1}^d \widetilde{C_i}$ be 
the strict transform of $D$ in $\widetilde{X}$ and let $E_p:= \tau^{-1}(p)$ be the exceptional divisor over the point $p$.

Note that the singular locus of $Z$ is precisely the pre-image, under
$\pi$, of the singular points of $\mathcal{C}$ of multiplicity at
least 3 (see \cite[Proposition 3.1]{Pa}, for example). 
Since $\tau$ is defined to be the blow up of 
the singular points of $\mathcal{C}$ of multiplicity at
least 3, there exists a morphism
$\sigma:Y \to \widetilde{X}$, by the universal property of blow ups.  
See the commutative diagram in Figure 
\ref{dia:diagram1}. 

From the commutativity of the diagram, it is easy to see that $\sigma$ is also
an abelian cover with Galois group
$(\mathbb{Z}/2\mathbb{Z})^{d-1}$, branch divisor $\widetilde{D}$
and ramification index 2 at every irreducible component of $\widetilde{D}$. Then 
$\sigma^{\star}E_p$ is a divisor in $Y$ consisting of
$2^{d-1-r_p}$ disjoint curves $F_p$, 
each with multiplicity 2. See \cite[II.3.2]{Hi} for more details. 
For a point 
$x \in E_p$ which is not in the branch locus of $\sigma$, $\sigma^{-1}(x)$ consists of
$2^{d-1}$ distinct points and these are contained in the $2^{d-1-r_p}$
disjoint curves $F_p$. Since each $F_p$ occurs with multiplicity 2 in
$\sigma^{\star}E_p$, the number of elements in a single $F_p$ that map to $x$ is
$\frac{2^{d-1}}{2(2^{d-1-r_p})} = 2^{r_p-1}.$ So each $F_p$ is a
finite cover of $E_p$ of degree $2^{r_p-1}$.  The branch locus of the
map $F_p \to E_p$ is precisely the $r_p$ intersection points of 
$E_p$ and $\widetilde{D}$. Since the ramification index is 2 and the
degree of the map $F_p \to E_p$ is $2^{r_p-1}$, there are
$\frac{2^{r_p-1}}{2} = 2^{r_p-2}$ points in $F_p$ that map to any
point in the branch locus. Hence the degree of the ramification
divisor is $2^{r_p-2}r_p$.



\begin{figure}[h] 
\[ \xymatrix{
   Y \ar[r]^{\rho}\ar[d]_{\sigma} & Z\ar[d]^{\pi}\\
  \widetilde{X}  \ar[r]^{\tau}       & D
 } \]
\caption{Construction of the surface $Y$}\label{dia:diagram1}
\end{figure}


By the above discussion, we have $\sigma^{\star}E_p = \sum 2 F_p$ with
$2^{d-1-r_p}$ terms in the summand. So
$$-2^{d-1} = 2^{d-1}(E_p)^2= (\sigma^{\star}E_p)^2=4(2^{d-1-r_p})F_p^2,$$
which implies that $F_p^2=-2^{r_p-2}$ for every point $p \in
\text{Sing}(\mathcal{C})$ with $r_p \ge 3$.

Using the Hurwitz formula to compute the Euler characteristic of
$F_p$, we get 
\begin{equation}\label{eq:2-2g}
e(F_p) = 2-2g(F_p)=2^{r_p-1}(2) - 2^{r_p-2}r_p = 2^{r_p-2}(4-r_p). 
\end{equation}

 We will calculate the Chern numbers $c_2$,\ $c_1^2$ of $Y$, where
 $c_2$ is same as the Euler characteristic $e(Y)$ of $Y$ and $c_1^2$ is the self-intersection number of a canonical divisor of $Y$.

Note that 
\begin{align*} 
Y\setminus \bigcup\limits_{p,r_p\geq3}\sigma^{-1}E_{p}=(\tau\circ \sigma)^{-1}\left((X\setminus\mathcal{C})\cup(\mathcal{C}\setminus\text{Sing}(\mathcal{C}))\cup\{p\in \text{Sing}(\mathcal{C})| r_p= 2 \}\right).
\end{align*}
 
If $A \to B$ is an \'etale map of degree $n$, then
$e(A)=ne(B)$. Since $\sigma$ is an \'etale map on 
$Y\setminus \bigcup\limits_{p,r_p\geq3}\sigma^{-1}E_{p}$,
 we get
\begin{equation}\label{eq:e1}
e\left(Y\setminus \bigcup\limits_{p,r_p\geq3}\sigma^{-1}E_{p}\right)=2^{d-1}e(X\setminus\mathcal{C})+
2^{d-2}e(\mathcal{C}\setminus\text{Sing}(\mathcal{C}))+2^{d-3}t_{2}.
\end{equation}
Using the additivity of the topological 
Euler characteristic, we have the following: 
\begin{enumerate}
\item[] $e(\mathcal{C})=2\sum(1-g(C_i))- \sum_{k\geq2}(k-1)t_{k},$
\item[]
\item[]
  $e(\mathcal{C}\setminus\text{Sing}(\mathcal{C}))=2\sum(1-g(C_i))-\sum_{k\geq2}kt_{k},$
\item[]
\item[] $e(X\setminus\mathcal{C})=e(X)+2\sum(g(C_i)-1)+
  \sum_{k\geq2}(k-1)t_{k}.$
\end{enumerate}
Substituting these values in \eqref{eq:e1}, we have 
\begin{align*}
e\left(Y\setminus
  \bigcup\limits_{p,r_p\geq3}\sigma^{-1}E_{p}\right)=~&2^{d-1}\left(e(X)+2\sum(g(C_i)-1)+
                                                      \sum_{k\geq2}(k-1)t_{k}\right)+\\
~&2^{d-2}\left(-2\sum(g(C_i)-1) -\sum_{k\geq2}kt_{k}\right)+2^{d-3}t_{2}.
\end{align*}
It is easy to check that $$e(X)=4-4g \text{~and~} 2g(C_{i})-2
=-a^2e+2ab+ae+a(2g-2)-2b.$$
Note also that $\sum_{k\geq2}(k-1)t_{k}=f_1-f_0$.

So we get 
\begin{align*}
e\left(Y\setminus \bigcup\limits_{p,r_p\geq3}\sigma^{-1}E_{p}\right)= ~&2^{d-1}\left(4-4g+d(-a^2e+2ab+ae+a(2g-2)-2b)+f_1-f_0\right)+\\
~ &2^{d-2}\left(-d(-a^2e+2ab+ae+a(2g-2)-2b)-f_1\right)+2^{d-3}t_{2}.
\end{align*}
There are $2^{d-1-r_{p}}$ curves with Euler characteristic $e(F_{p})$ in $Y$
over each exceptional divisor $E_p$ in $\widetilde{X}$. So
\eqref{eq:2-2g} gives
\begin{align*}
e(Y)&=e\left(Y\setminus \bigcup\limits_{p,r_p\geq3}\sigma^{-1}E_{p}\right)+\sum_{k\geq3}2^{d-1-k}t_{k}e(F_p)\\
&=e\left(Y\setminus \bigcup\limits_{p,r_p\geq3}\sigma^{-1}E_{p}\right)+\sum_{k\geq3}2^{d-1-k}t_{k}\left(2^{k-1}(2-k)+k2^{k-2}\right)\\
&=e\left(Y\setminus \bigcup\limits_{p,r_p\geq3}\sigma^{-1}E_{p}\right)+2^{d-3}\sum_{k\geq3}t_{k}(4-k)\\
&=e\left(Y\setminus \bigcup\limits_{p,r_p\geq3}\sigma^{-1}E_{p}\right)+2^{d-3}(4f_0-f_1-2t_2)\\
\end{align*}
Now using the value of $e\left(Y\setminus \bigcup\limits_{p,r_p\geq3}\sigma^{-1}E_{p}\right)$ computed above and  simplifying, we get
\begin{equation}\label{eq:e(Y)}
\frac{1}{2^{d-3}} e(Y)
= 16-16g+d(-2a^2e+4ab+2ae+4ag-4a-4b)+ f_{1}-t_{2}.
\end{equation}

Next we calculate $c_1^2(Y).$
 
For the divisor $D=\sum_{i=1}^{d}C_i$ on $X$, we know that $\tau^{\star}D -
 \sum\limits_{\substack{p\in \text{Sing}(\mathcal{C}),\\
     r_p\geq3}}r_pE_p$ is the strict
 transform of $D$ in $\widetilde{X}.$  The divisors $\sigma^{\star}(\tau^{\star}D
 - \sum r_pE_p)$ and $\sigma^{\star}E_p$ of $Y (p\in
 \text{Sing}(\mathcal{C}),r_p\geq3)$ are divisible by $2$. For a
 canonical divisor $K_X$ of $X$, $\tau^{\star}K_X+\sum E_p$ is a canonical
 divisor of $\widetilde{X}.$ Applying 
\cite[Page 42, Lemma 17.1]{Bar} to the ramified covering $ \sigma:Y \to \widetilde{X},$ we get the following:
 
 \begin{lemma}\label{Lemma 3.2} Let $Y$ be the surface constructed in
   Figure \ref{dia:diagram1}. 
 	The canonical divisor of $Y$ is given by $K_{Y}=\sigma^{\star}T$ for the $\rationals$-divisor $T$ on $\widetilde{X}$ defined as
 	$$T:=\tau^{\star}K_{X}+\sum E_{p}+\frac{1}{2}\left(\sum E_{p}+\tau^{*}D-\sum r_{p}E_{p}\right),$$
 	where the summations are taken over all the points $p \in
        \text{Sing}(\mathcal{C})$ such that $r_p\geq3$.
\end{lemma}
 Thus, 
$T^2=K_{X}^2+K_{X}\cdot
D-\sum_{k\geq3}t_{k}+\sum_{k\geq3}(k-1)t_{k}+\frac{1}{4}(D^2-\sum_{k\geq3}(k-1)^2t_{k}).$

We have the following: 
\begin{enumerate}
\item[] $K_{X}^2=8(1-g)$,
\item[] $K_{X}\cdot D=d\left(ae+a(2g-2)-2b\right)$,
\item[] $\sum_{k\geq3}t_{k}=f_{0}-t_{2}$, $\sum_{k\geq3}(k-1)t_{k}=f_{1}-f_{0}-t_{2}$, and
\item[]  
\item[] $D^2-\sum_{k\geq3}(k-1)^2t_{k}=d(-a^2e+2ab)+f_{1}-f_{0}+t_{2}$. For
  this equality, use Lemma \ref{eq:combinatorial equality1}(2).
\end{enumerate}

Substituting these values in the expression for $T^2$ and noting that $c_1^2(Y)=2^{d-1}T^2,$ we get:
\begin{equation}\label{eq:c_1^2(Y)}
\frac{1}{2^{d-3}}c_1^2(Y)
=32-32g+d(-a^2e+2ab+4ae+8ag-8a-8b)-9f_0+5f_1+t_{2}.\\
\end{equation}
Now we have, by \eqref{eq:e(Y)} and \eqref{eq:c_1^2(Y)},
\begin{equation}\label{3e(Y)-c_1^2(Y)}
\frac{1}{2^{d-3}}(3e(Y)- c_1^2(Y))=16-16g+d[(2b-ae)(5a-2)+4a(g-1)]+9f_0-2f_1-4t_2.\\
\end{equation}

\begin{remark}\label{No.of curves}
	By \eqref{eq:2-2g}, $F_p$ is rational if and only if $r_p=3$ and $F_p$  is elliptic if and only if $r_p=4.$ Thus we know that $Y$ contains $2^{d-4}t_{3}$ disjoint $(-2)$-curves (above the $3$-points) and contains $2^{d-5}t_{4}$ elliptic curves (above the $4$-points), each of self-intersection $-4$.
\end{remark}
\section{Harbourne Constants}\label{main-results} 
In this section, we will first show that the surface $Y$ (constructed
in the last
section; see Figure \ref{dia:diagram1}) has non-negative Kodaira
dimension. This will allow us to apply a Hirzebruch-Miyaoka-Sakai
inequality involving the Chern numbers of $Y$ and certain curves on
$Y$ coming from the arrangement $\mathcal{C}$ on $X$ (see
Theorem \ref {Theorem 4.6}). Using this we obtain a Hirzebruch-type
inequality \eqref{eq:Hirzebruch type ineq}. We prove our bound for the Harbourne
constant of $\mathcal{C}$ in Theorem \ref{Theorem 4.7}. 

We will use the notation of Section \ref{abelian-cover}. Recall that
$T$ is a $\rationals$-divisor on $\widetilde{X}$ defined in Lemma
\ref{Lemma 3.2}. We start with the following.

\begin{lemma}\label{eq:Lemma 4.1}
Let $X$ be a ruled surface with $e \ge 4$. Let $\mathcal{C}$ be a
transversal arrangement of curves satisfying 
Assumption
\ref{star}. Then $T\cdot E_{p}\geq 0$ for every $p \in
\text{Sing}(\mathcal{C})$ such that $r_p\geq3$. 
\end{lemma}
\begin{proof}
	$T\cdot E_{p}=-1+\frac{-1+r_{p}}{2}\geq -1+ \frac{-1+3}{2}=0.$
\end{proof}
\begin{lemma}\label{eq:Lemma 4.2}
	Let $X$ be a ruled surface with $e \ge 4$. Let $\mathcal{C}$ be a
transversal arrangement of curves satisfying 
Assumption
\ref{star}.          Let $C_j'=\tau^{\star}C_j-\sum\limits_{p\in
  C_j,r_p\geq3}E_p$ be the strict transform of $C_j \in
\mathcal{C}$, for $j=1,2,\ldots,d$. Then  $ T\cdot C_j'\geq 0.$
\end{lemma}
\begin{proof}
Let $f_{0}^j$ denote the number of multiple points on $C_j$ and let $t_{k}^j$ denote the number of $k$-fold points on $C_j.$

Now,
\begin{equation}\label{eq:eq 4.1}
T\cdot C_j'= K_X\cdot C_j+\frac{D\cdot C_j}{2}-T\cdot\sum_{p\in C_j,r_p\geq3}E_p.
\end{equation}
We now compute each of the terms individually. 
\begin{align*}
K_X\cdot C_j&=2ae-2b+(2g-2-e)a,\\
 D\cdot C_j&= d(2ab-a^2e),\\
T\cdot E_{p}&=\frac{r_p-3}{2} ; \ \ p\in C_j,r_p\geq3.\\
\end{align*}
By Lemma \ref{eq:combinatorial equality1} (1), we have
\begin{align*}
T\cdot\sum_{p\in C_j,r_p\geq3}E_p&=
\sum_{p\in C_j,r_p\geq3}\frac{r_{p}-3}{2}\\
&=\sum_{p\in C_j,r_p\geq2}\frac{r_{p}-1}{2}-f_{0}^j+\frac{t_{2}^j}{2}\\
&=\frac{(2ab-a^2e)(d-1)}{2}-f_{0}^j+\frac{t_{2}^j}{2}.
\end{align*}
Plugging the values computed above in \eqref{eq:eq 4.1}, we get
\begin{eqnarray}\label{4.2}
T\cdot C_j'=2ae-2b+(2g-e-2)a+\frac{2ab-a^2e}{2} + f_{0}^j-\frac{t_{2}^j}{2}.
\end{eqnarray}
To prove the lemma, it suffices to show 
\begin{equation}\label{eq:eq 4.2}
f_{0}^j-\frac{t_{2}^j}{2} \geq -\left(\frac{2ab-a^2e}{2}\right)-2ae+a(e+2)+2b.
\end{equation}
Let $k$ be the maximum of the multiplicities	 of the points on $C_j$.
By Lemma \ref{eq:combinatorial equality1} (1),  we have 
$$t_{2}^j+2t_{3}^j+\ldots+(k-1)t_k^j=(2ab-a^2e)(d-1).$$

Now,
\begin{align*}
f_{0}^j-\frac{t_{2}^j}{2}&=\frac{t_{2}^j}{2}+t_{3}^j+\ldots+t_{k}^j\\ &\geq \frac{t_{2}^j+2t_{3}^j+\ldots+(k-1)t_k^j}{k}=\frac{(2ab-a^2e)(d-1)}{k}\\&\geq 2ab-a^2e,
\end{align*} where last inequality holds since $k\leq d-1.$

Thus in order to show \eqref{eq:eq 4.2}, it suffices to show  the
following inequality:  
\begin{equation}\label{eq:eq 4.3}
2ab-a^2e \geq -\left(\frac{2ab-a^2e}{2}\right)-2ae+a(e+2)+2b.
\end{equation}

Now we have the following: 
\begin{eqnarray*}
\eqref{eq:eq 4.3} &\Leftrightarrow&  6ab-4a-4b \geq 3a^2e-2ae\\
&\Leftarrow& b\geq\frac{4a}{3a-2}\\
&\Leftarrow& ae \geq \frac{4a}{3a-2}\\
& \Leftarrow& e \ge 4. 
\end{eqnarray*}
The last inequality holds by Assumption \ref{star}. 
\end{proof}
We now make a further assumption on our arrangement
$\mathcal{C}$. This is required for our argument showing that $K_Y$ is
nef. 

\begin{assumption}\label{extra-assumption}
Let $X$ be a ruled surface over a smooth curve with $e \ge 4$. 
Let $\mathcal{C}$ be a transversal arrangement of curves on a ruled
surface $X$
satisfying Assumption \ref{star}. Assume further that $\mathcal{C}$ satisfies one of the
following conditions: 
\begin{enumerate}
\item $a\geq 2$, or
\item $a=1$ and there exists a subset of four curves in $\mathcal{C}$
  such that there is no point common to all the four curves. 
\end{enumerate}
\end{assumption}

\begin{question}
We do not know any example of a transversal arrangement for which
Assumption \ref{extra-assumption} does not hold. Does this 
assumption always hold for any arrangement satisfying Assumption
\ref{star}?
\end{question}

\begin{theorem}\label{Theorem 4.5} Let $X$ be a ruled surface with $e
  \ge 4$ and let $\mathcal{C}$ be a transversal arrangement of curves
  satisfying Assumption \ref{extra-assumption}. Let $Y$ be the surface
  constructed in Figure \ref{dia:diagram1}. Then 
$K_Y$ is nef. 
\end{theorem}
\begin{proof}
Recall (see Lemma \ref{Lemma 3.2}) that $T$ is a divisor on $\widetilde{X}$ given by 
\begin{equation}\label{eq:eq 4.5}
T:=\tau^{\star}K_{X}+\frac{3}{2}\sum_{r_p\geq 3} E_{p}+\frac{1}{2}\sum C_i',
\end{equation} 
where $C_i'$ is the strict transform of $C_i$ by $\tau$ and
$E_p=\tau^{-1}(p).$ Note that $K_Y = \sigma^{\star}T$. 
We have $\tau^{\star}C_i=C_i'+\sum_{p \in C_i,r_p\geq 3} E_{p}.$ 

We want to express $T$ as a positive sum of effective divisors on
$\widetilde{X}.$ The negative terms in the expression occur because of
the term involving $K_X = -2C_0+(2g-2-e)f.$
We consider two different cases.

\textbf{Case (1):} Assume $a \geq 2.$
   Let $C_1, C_2 \in \mathcal{C}$. 

For $q:=a-2\geq 0, p:=2g-e-2+b \geq 0$, we have
$K_X = pf+qC_0-\frac{C_1+C_2}{2}$. Note that $p >  0$, since $b \ge
ae$ and $e \ge 4$. 

Thus, \eqref{eq:eq 4.5} becomes,
\begin{eqnarray*}
T&=&\tau^{\star}(pf+qC_0)-\frac{1}{2}\left(C_1'+\sum_{p \in C_1,r_p\geq 3} E_{p}+C_2'+\sum_{p \in C_2,r_p\geq 3} E_{p}\right)+\frac{3}{2}\sum_{r_p\geq 3} E_{p}+\frac{1}{2}\sum_{i=1}^{d} C_i'\\
&=&\tau^{\star}(pf+qC_0)+\frac{1}{2}\sum_{i=3}^{d}
    C_i'+\sum\lambda_{p}E_{p},  \text{~for some~} \lambda_{p}.
\end{eqnarray*} 

Note that $\lambda_p$
is non-negative for every point $p \in \text{Sing}(\mathcal{C})$ with $r_p \ge
3$. Indeed, $\lambda_p = \frac{3}{2}$ if $p \notin C_1 \cup C_2$;
$\lambda_p = 1$ if $p$ belongs to exactly one of the curves $C_1$ or
$C_2$; and $\lambda_p = \frac{1}{2}$ if $p \in C_1 \cap C_2$.
Thus $T$ is effective and we have 
$$K_Y = \sigma^{\star}T = \sigma^{\star}\tau^{\star}(pf+qC_0)+\sigma^{\star}\left(\frac{1}{2}\sum_{i=3}^{d} C_i'\right)+\sigma^{\star}(\sum\lambda_{p}E_{p}).$$
If $C$ is a curve in $Y$ not contained in $\sigma^{\star}E_{p}$ and $\sigma^{\star}C_i',$
$$K_Y\cdot C = 0 + \sigma^{\star}\left(\frac{1}{2}\sum_{i=3}^{d} C_i'\right)\cdot C+\sigma^{\star}\left(\sum\lambda_{p}E_{p}\right)\cdot C \geq 0.$$

If $C$ is a curve in $Y$ such that $C$ is either $\sigma^{\star}C_i'$
or in $\sigma^{\star}E_{p},$ 
Lemma \ref{eq:Lemma 4.1} and Lemma \ref{eq:Lemma 4.2} imply that $K_Y\cdot C \geq 0.$ 
Thus $K_Y\cdot C\geq 0$ for every curve $C$ in $Y$.
Hence, $K_Y$ is nef.

\textbf{Case (2):} Suppose that $a=1$. By
Assumption \ref{extra-assumption}, there are four
curves, say $C_1, C_2, C_3, C_4$, in $\mathcal{C}$ such that no point is contained in all the four curves.

Let $p:=2g-2-e+2b > 0$. Then $K_X = pf-\frac{C_1+C_2+C_3+C_4}{2}$.

Thus,
\begin{align*}
T&=\tau^{\star}(pf)-\frac{1}{2}\left(\sum_{i=1}^{4}C_i'+\sum_{p \in C_i,r_p\geq 3} E_{p}\right)+\frac{3}{2}\sum_{r_p\geq 3} E_{p}+\frac{1}{2}\sum_{i=1}^{d} C_i'.\\
&=\tau^{\star}(pf)-\frac{1}{2}\left(\sum_{p \in C_i,r_p\geq 3} E_{p}\right)+\frac{3}{2}\sum_{r_p\geq 3} E_{p}+\frac{1}{2}\sum_{i=5}^{d} C_i'.\\
&=\tau^{\star}(pf)+\frac{1}{2}\sum_{i=5}^{d}
  C_i'+\sum\lambda_{p}'E_{p}, \text{~for some~} \lambda_{p}'.
\end{align*}

We have $\lambda_p' = \frac{3}{2}$ if $p \notin C_1 \cup C_2 \cup
C_3 \cup C_4$. By Assumption \ref{extra-assumption} and the choice of $C_1, C_2, C_3, C_4$, there are no
points in the intersection $C_1 \cap C_2 \cap C_3 \cap C_4$. If $p$
belongs to three of them, then $\lambda_p' = \frac{3}{2} - \frac{3}{2}
= 0$. So we have $\lambda_p' \ge 0$ for all $p \in \text{Sing}(\mathcal{C})$ with $r_p \ge
3$.  Thus $T$ is effective and we have 
$$K_Y=\sigma^{\star}\tau^{\star}(pf)+\frac{1}{2}\sigma^{\star}\left(\sum_{i=5}^{d} C_i'\right)+\sigma^{\star}\left(\sum\lambda_{p}'E_{p}\right). $$

If $C$ is a curve in $Y$ not contained in  $\sigma^{\star}E_{p}$ and $\sigma^{\star}C_i',$
$$K_Y\cdot C = 0 + \sigma^{\star}\left(\frac{1}{2}\sum_{i=5}^{d} C_i'\right)\cdot C+\sigma^{\star}\left(\sum\lambda_{p}'E_{p}\right)\cdot C \geq 0.$$

If $C$ is a curve in $Y$ such that $C$ is either $\sigma^{\star}C_i'$
or in $\sigma^{\star}E_{p},$ Lemma \ref{eq:Lemma 4.1} and Lemma \ref{eq:Lemma 4.2} imply that $K_Y\cdot C \geq 0.$ Thus $K_Y\cdot C\geq 0$ for every curve $C$ in $Y.$
Hence, $K_Y$ is nef.
\end{proof}

The following result of Hirzebruch \cite[Theorem 3, Page 144]{H2} is crucial in our computations. It
strengthens earlier results of Miyaoka and Sakai. 

\begin{theorem}[Hirzebruch]\label{Theorem 4.6}
	Let $X$ be a smooth surface of general type and $E_{1}, \ldots,
        E_{k}$ configurations (disjoint to each other) of rational
        curves on  $X$ (arising from quotient singularities) and
        $C_{1}, \ldots, C_{p}$  smooth elliptic curves (disjoint to each
        other and disjoint to the $E_{i}$).  Let $c_{1}^{2}(X), c_{2}(X)$ be the Chern numbers of $X$. Then
	$$3c_{2}(X) - c_{1}^{2}(X) \geq \sum_{j=1}^{p}(-C_{j}^{2}) + \sum_{i=1}^{k}m(E_{i}).$$
\end{theorem}

Hirzebruch in fact remarks that the result also holds when $X$ has
non-negative Kodaira dimension. We use the theorem in this case. 

The numbers $m(E_{i})$ mentioned in the theorem are positive numbers
defined using certain invariants (Euler characteristics, self-intersections) of
the arrangements $E_i$. Hirzebruch gives a formula to compute them
in \cite[Page 144, (5)]{H2} which shows that 
if $E_{i}$ is a single $(-2)$-curve, then  $m(E_{i})=\frac{9}{2}.$
See also \cite{He}. 

Now we are ready to prove the main result of this paper.

\begin{theorem}\label{Theorem 4.7}
Let $X$ be a ruled surface with $e
  \ge 4$ over a smooth curve of genus $g$.  Let $\mathcal{C}$ be a transversal arrangement of curves
  satisfying Assumption \ref{extra-assumption}. In particular, each curve in
  $\mathcal{C}$ is numerically equivalent to $aC_0+bf$ with $a > 0$
  and $b \ge ae$. 
Then we have the
  following bound on the Harbourne constant of $\mathcal{C}$:
\begin{equation}\label{eq:H const1}
H(X,\mathcal{C}) \geq \frac{-9}{2}-\frac{8}{f_{0}}+\frac{d}{f_{0}}\left(\frac{(ae-2b)}{2}(3a-2)-2a(g-1)\right)+\frac{16g+4t_2+t_4}{2f_0}+\frac{9t_3}{8f_0}.
\end{equation}
\end{theorem}
\begin{proof}
By Remark \ref{No.of curves}, the surface $Y$ (constructed in Figure
\ref{dia:diagram1}) contains $2^{d-4}t_{3}$ disjoint rational $(-2)$-curves $E_i$
(above the $3$-points) and contains $2^{d-5}t_{4}$ elliptic
curves $C_j$ (above the $4$-points), each of self-intersection $-4$.

By Theorem \ref{Theorem 4.5}, $K_Y$ is nef. Thus, by Theorem \ref {Theorem 4.6}:
\begin{equation}\label{eq:Miyaoka-Sakai}
\frac{3c_{2}(Y) - c_{1}^{2}(Y)}{2^{d-3}} \geq \frac{ \sum(-C_{j}^{2}) + \sum m(E_{i})}{2^{d-3}}. 
\end{equation}

As noted earlier,
$m(E_{i})=\frac{9}{2}$ for all rational curves $E_{i}$ of self-intersection $-2.$

From \eqref{3e(Y)-c_1^2(Y)}, we have,
\begin{align*}
\frac{1}{2^{d-3}}(3e(Y)- c_1^2(Y))=16-16g+d[(2b-ae)(5a-2)+4a(g-1)]+9f_0-2f_1-4t_2.
\end{align*}

Also, from our discussion above,we have  
\begin{align*}
\sum m(E_i)&=\frac{9}{2}2^{d-4}t_{3}, \text{~and}  & \\
\sum (-C_j^2)&= 4t_42^{d-5}.
\end{align*}

Plugging these values in \eqref {eq:Miyaoka-Sakai} and simplifying,
we  have :
\begin{equation}\label{eq:4.7}
16-16g+d(2ae-5a^2e+10ab+4ag-4a-4b)+9f_0-2f_1-4t_2 -t_{4} - \frac{9}{4}t_{3}\geq 0.
\end{equation}

Simplifying and re-arranging \eqref{eq:4.7}, we  obtain the following Hirzebruch-type inequality for $\mathcal{C}$:\\
\begin{equation}\label{eq:Hirzebruch type ineq}
t_{2}+\frac{3}{4}t_{3} \geq -16+16g+\sum_{k\geq5}(2k-9)t_{k}+
d(e(5a^2-2a)-10ab-4ag+4a+4b).\\
\end{equation}

Now we bound $H(X, \mathcal{C})$. 
We have
\begin{align*}
H(X,\mathcal{C})&=\frac{(2ab-a^2e)d^2-\sum_{k\geq 2}k^2t_{k}}{f_0}
&=\frac{(2ab-a^2e)d^2-f_2}{f_0}
&=\frac{(2ab-a^2e)d-f_1}{f_0},
\end{align*}
where the last equality follows from  Lemma \ref{eq:combinatorial equality1}(2).

From \eqref{eq:4.7}, we have $$-f_{1} \geq \frac{-16+16g+d\left(e(5a^2-2a)-10ab-4ag+4a+4b\right)-9f_{0}+4t_{2}+\frac{9}{4}t_{3}+t_{4}}{2}.$$

Thus,
\begin{align*}
	H(X,\mathcal{C}) &\geq \frac{d(-a^2e+2ab)-8+8g+\frac{d(e(5a^2-2a)-10ab-4ag+4a+4b)-9f_{0}}{2}+2t_{2}+\frac{9}{8}t_{3}+\frac{t_{4}}{2}}{f_{0}}\\
	&=
	\frac{-9}{2}-\frac{8}{f_{0}}+\frac{d}{f_{0}}\left(\frac{ae}{2}(3a-2)-2ag-3ab+2a+2b\right)+\frac{16g+4t_2+t_4}{2f_0}+\frac{9t_3}{8f_0}. 
\end{align*} This completes the proof of the theorem. 
\end{proof}

If the curves in the arrangement $\mathcal{C}$ do not intersect the
normalized section $C_0$, then we obtain an improved bound for
the $H$-constants as shown in the following proposition. We obtain an
improved bound in this case because $Y$ contains some additional
rational curves. 

\begin{proposition}\label{Prop 3.6}
Let $X$ be a ruled surface with $e
  \ge 4$ over a smooth curve of genus $g$.  Let $\mathcal{C}$ be a transversal arrangement of curves
  satisfying Assumption \ref{extra-assumption}.
Assume further that no curve in $\mathcal{C}$ intersects the
normalized section $C_0$. Then we have the following bound on the
Harbourne constant of $\mathcal{C}$: 
\begin{equation}\label{eq:Hconst2}
H(X,\mathcal{C}) \geq \frac{-9}{2}+\frac{d}{f_{0}}\left(\frac{ae(2-3a)-4a(g-1)}{2}\right)+\frac{16g+4t_2+t_4}{2f_0}+\frac{9t_3}{8f_0}.
\end{equation}
\end{proposition}
\begin{proof}
As in the previous theorem, by Remark \ref{No.of curves}, 
the surface  $Y$ contains $2^{d-4}t_{3}$ disjoint rational $(-2)$-curves $E_i$
(above the $3$-points), $2^{d-5}t_{4}$ elliptic
curves $C_j$ (above the $4$-points), each of self-intersection
$-4$. Further, since 
the curves in the arrangement do not intersect $C_0$, the surface
$\tilde{X}$ has an isomorphic copy of $C_0$. Hence 
$Y$ contains $2^{d-1}$ copies of a rational curve $H$ of
self-intersection $-e.$

Hirzebruch gives a formula to compute the value $m(H)$ in \cite[Page
144, (4)]{H2}. Applying this formula, we have that for  rational
curves  $H$ of self-intersection $-e$, $m(H) = 2+e+\frac{1}{e}.$

By Theorem \ref{Theorem 4.5}, $K_Y$ is nef. Thus,  by Theorem \ref{Theorem 4.6}, the inequality in \eqref {eq:Miyaoka-Sakai} is satisfied.

From \eqref{3e(Y)-c_1^2(Y)}, we have,
\begin{align*}
\frac{1}{2^{d-3}}(3e(Y)- c_1^2(Y))=16-16g+d[(2b-ae)(5a-2)+4a(g-1)]+9f_0-2f_1-4t_2.
\end{align*}

We have 
\begin{align*}
\sum m(E_i) + \sum m(H) &=\frac{9}{2}2^{d-4}t_{3}+2^{d-1}(2+e+\frac{1}{e}), \text{~and} & \\
\sum (-C_j^2)&= 4t_42^{d-5}.
\end{align*}
Plugging these values in  \eqref{eq:Miyaoka-Sakai} and simplifying, we have:
\begin{equation}\label{eq:4.10}
16-16g+d(2ae-5a^2e+10ab+4ag-4a-4b)+9f_0-2f_1-4t_2 -t_{4} - \frac{9}{4}t_{3}-4(2+e+\frac{1}{e})\geq 0.
\end{equation}
Simplifying \eqref {eq:4.10}, with $ae=b,$ we arrive at the following modified Hirzebruch-type  inequality for $\mathcal{C}$ :
\begin{equation}\label{eq:hirz}
t_{2}+\frac{3}{4}t_{3} \geq 4(e+\frac{1}{e})-8+16g+\sum_{k\geq5}(2k-9)t_{k}+d(-5a^2e+2ae-4ag+4a).
\end{equation}
Since $e\geq 4$, we have $4(e+\frac{1}{e}) \geq 17$. So
\eqref {eq:hirz} becomes:
\begin{equation}\label{eq:eq 4.12}
t_{2}+\frac{3}{4}t_{3} \geq 9+16g+\sum_{k\geq5}(2k-9)t_{k}+d\left(-5a^2e+2ae-4ag+4a\right).
\end{equation}
From the above inequality \eqref {eq:eq 4.12}, we have $$-f_{1} \geq \frac{9+16g+d\left(e(2a-5a^2)-4ag+4a\right)-9f_{0}+4t_{2}+\frac{9}{4}t_{3}+t_{4}}{2}.$$

We now bound the $H$-constant $H(X,\mathcal{C})$. 
\begin{align*}
H(X,\mathcal{C}) &\geq \frac{d(-a^2e+2ab)+8g+\frac{d(e(2a-
		5a^2)-4ag+4a)-9f_{0}+9}{2}+2t_{2}+\frac{9}{8}t_{3}+\frac{t_{4}}{2}}{f_{0}}\\
&\geq \frac{-9}{2}+\frac{d}{f_{0}}\left(\frac{-7a^2e}{2}+2ab+ae-2ag+2a\right)+\frac{16g+4t_2+t_4}{2f_0}+\frac{9t_3}{8f_0}.
\end{align*}
Since $ae=b$, we get 
$$H(X,\mathcal{C}) \geq
\frac{-9}{2}+\frac{d}{f_{0}}\left(\frac{ae(2-3a)-4a(g-1)}{2}\right)+\frac{16g+4t_2+t_4}{2f_0}+\frac{9t_3}{8f_0},$$
as required.
\end{proof}

We now define the $H$-constant of a ruled surface for a fixed pair of
integers $a,b$ as follows. 

\begin{definition}\label{global-harbourne}
Let $X$ be a ruled surface with invariant $e \ge 4$. Let $a > 0$ and
$b \ge ae$ be positive integers. 
	We define the $H$-constant $H_{a,b}(X)$ of $X$ as :
	$$H_{a,b}(X):=\inf_{\mathcal{C}} H(X,\mathcal{C}),$$
	 where the infimum is over all transversal arrangements $\mathcal{C}$ satisfying Assumption \ref{extra-assumption}.
\end{definition}

In order to bound the constant $H_{a,b}(X)$, we make the following
observation. 
\begin{lemma}\label{Lemma 4.10}
	Let $\mathcal{C}=\{C_1,C_2,\ldots,C_d\}$ be a transversal arrangement on the ruled surface $X$ satisfying Assumption \ref{extra-assumption}. Then $f_0\geq d.$ 
\end{lemma}
\begin{proof}
This is proved in \cite[Lemma 6.1]{E}. We write the proof here for the
convenience of the reader.

Let $s=f_0$ and $h=2ab-a^2e.$ Let $\text{Sing}(\mathcal{C})=\{p_1,\ldots,p_s\}$. 
Consider the $\rationals-$vector space $\mathbb{Q}^s$ with the usual
dot product: if $v=(a_1,\ldots,a_s)$ and $u = (b_1,\ldots,b_s)$, then
$v\cdot u: = a_1b_1+\ldots+a_sb_s$. 
	
For every curve $C_i \in \mathcal{C}$, we associate a vector 
$v_i \in \rationals^s$ by setting the $l$-th entry of $v_i$ equal to
1, if  $C_i$ passes through $p_l$, and 0 otherwise.

Note that if $i\neq j$, then $v_i\cdot v_j$ is precisely the number of
points common to $C_i$ and $C_j.$ 
By our hypothesis, we have $v_i\cdot v_j=h$.
Also  $v_i\cdot v_i$ is the number of multiple points that are contained in $C_i$.

We claim that each curve $C_i$ contains at least $h+1$ intersection
points with other curves in the arrangement. 
Since there are at least two curves in $\mathcal{C}$, we have $v_i \cdot v_i \ge h$. 
If $v_i \cdot v_i = h$, then all the curves in $\mathcal{C}$ intersect
$C_i$ in the same $h$ points. This contradicts the assumption $t_d =
0$.  Thus $v_i\cdot v_i>h$ for all $i$.

To prove the lemma, it suffices to show that the set
$\{v_1,v_2,\ldots,v_d\}$ is linearly independent. 
If it is not linearly independent, without loss of generality, let 
$v_1=\sum_{j=2}^{d}a_jv_j$ for $a_j \in \mathbb{Q}.$
	
Consider $v_1\cdot (v_1-v_q)$ where $q\geq 2.$
Then 
\begin{align*}
(v_1\cdot v_1)-h&=v_1\cdot (v_1-v_q)\\
&=\left(\sum_{j=2}^{d}a_jv_j\right) \cdot (v_1-v_q)\\&=\sum_{j=2}^{d}a_j\Big{(}h-(v_j\cdot v_q)\Big{)}\\&=a_q\left(h-(v_q\cdot v_q)\right)
\end{align*}
	
So $a_q=\frac{(v_1\cdot v_1)-h}{h-(v_q\cdot v_q)} < 0$. Since this
holds for all $q \ge 2$, $v_1$ is a linear combination of
$v_2,\ldots,v_d$ with negative coefficients. 
But the entries of $v_i$ for any $i=1,\ldots,d$ are either 0 or 1 and we
obtain the required contradiction. 
\end{proof}
\begin{corollary}\label{main-corollary} Let $X$ be a ruled surface over a smooth curve of
  genus $g$ with invariant $e \ge 4$. Let $a > 0$ and $b > ae$ be positive integers. Then 
		\begin{equation}\label{eq:eq 4.13}
		H_{a,b}(X) \geq \frac{-11}{2}+\frac{(ae-2b)}{2}(3a-2)-2ag.
		\end{equation}
		Further, if $ae= b,$ then 
		\begin{equation}\label{eq:eq 4.14}
		H_{a,b}(X) \geq \frac{-9}{2}+\frac{ae(2-3a)-4ag}{2}.
		\end{equation}
\end{corollary}

\begin{proof} 
We first claim that $f_0>2ab-a^2e+1.$ Indeed, if not, $f_0\leq2ab-a^2e+1.$
Then 
\begin{eqnarray*}
(2ab-a^2e)d(d-1)&=&\sum_{k\geq2}k(k-1)t_{k}, \text{by Lemma \ref{eq:combinatorial equality1}(2)}\\
&\leq&(d-1)(d-2)f_0, \text{~since~} k \le d-1\\
 &\leq & (d-1)(d-2)(2ab-a^2e+1).
\end{eqnarray*}	 

This gives 
\begin{eqnarray*}
 (2ab-a^2e)d&\leq&(d-2)(2ab-a^2e+1) \\
\Rightarrow ~2(2ab-a^2e)&\leq&(d-2)\\
\Rightarrow ~2(d-1)&\leq& 2(2ab-a^2e)\leq(d-2), \text{~by Lemma
                             \ref{Lemma 4.10}}\\
\Rightarrow ~d & \le & 0. 
\end{eqnarray*}

This is a contradiction and the claim follows.
	 
Now, since $b>ae$, $e\geq4$ and $a>0$ by our assumptions, the claim gives 
$f_0\geq2ab-a^2e+2\geq2a(ae+1)-a^2e+2\geq4a^2+2a+2\geq8.$ Thus $f_0\geq8$ and hence $\frac{-8}{f_0}\geq-1.$
	
By Theorem \ref{Theorem 4.7}, we have $H(X,\mathcal{C}) \geq
\frac{-9}{2}-\frac{8}{f_{0}}+\frac{d}{f_{0}}\left(\frac{(ae-2b)}{2}(3a-2)-2a(g-1)\right)$. 
Note that $\frac{(ae-2b)}{2}(3a-2)-2ag$ is a negative number as $b>ae.$ Hence, as 
	$\frac{-8}{f_0}\geq-1$, Lemma \ref{Lemma 4.10} gives \eqref{eq:eq 4.13}.
	
	Similarly, by Proposition \ref{Prop 3.6}, we have
$H(X,\mathcal{C}) \geq
\frac{-9}{2}+\frac{d}{f_{0}}\left(\frac{ae(2-3a)-4a(g-1)}{2}\right)$. 
Since $\frac{ae(2-3a)-4ag}{2}$ is a negative number, Lemma \ref{Lemma 4.10} gives \eqref{eq:eq 4.14}.
\end{proof}
We now state a corollary which gives a lower bound on the
self-intersection of the strict
transform of the divisor associated to an arrangement of curves. 
\begin{corollary}\label{Cor 4.12}
	Let $\mathcal{C}$ be a transversal arrangement on
	the ruled surface $X$ satisfying  Assumption \ref {extra-assumption}. Let 
	$f: \widetilde{X} \to X$ be the blow-up of X at  $\text{Sing}
        (\mathcal{C})$. Let $\widetilde{D}$ denote the strict
        transform of $D$, which is the divisor defined as the sum of all the curves in $\mathcal{C}$. Then
	\begin{equation*}
	 \widetilde{D}^2 \geq-8-\frac{9}{2}s+d\left(\frac{(ae-2b)}{2}(3a-2)-2a(g-1)\right)+8g+2t_2+\frac{t_4}{2}+\frac{9t_3}{8}.
	\end{equation*}
	Further, if all curves in the arrangement do not intersect
        the normalized section $C_0,$ then
	\begin{equation*}
	\widetilde{D}^2 \geq \frac{-9}{2}s+d\left(\frac{ae(2-3a)-4a(g-1)}{2}\right)+8g+2t_2+\frac{t_4}{2}+\frac{9t_3}{8}.
	\end{equation*} 	
\end{corollary}
\begin{proof}
	Indeed, note that $f_0=s$ and $\widetilde{D}^2=sH(X,\mathcal{C}).$ The corollary now follows from \eqref{eq:H const1} and \eqref{eq:Hconst2}.
\end{proof}

\subsection{Examples} 
It is not easy to construct arrangements which have small Harbourne
constants. 
Most easy to construct examples of curve arrangements have much larger Harbourne
constants than our bounds predict. For example, if $\mathcal{C} =
\{C_1,\ldots, C_d\}$ is a
general arrangement of curves on a ruled surface $X$ satisfying our
assumptions, then it is easy to see that $H(X,\mathcal{C}) =
\frac{-2(d-2)}{d-1}$. Indeed, all singular points of $\mathcal{C}$
have multiplicity 2 and consequently, $t_2  = \binom{d}{2}C_1^2$ and $t_k =
0$ for $k \ge 3$. Now an easy calculation gives 
$H(X,\mathcal{C}) =\frac{-2(d-2)}{d-1}$. But this value is much larger
than the bounds given by our main results Theorem \ref{Theorem 4.7} or
Corollary \ref{main-corollary}. 

This situation is analogous to the case of line arrangements in
$\proj^2$. The best bound we have in this case is given in
\cite[Theorem 3.3]{B1} which proves that
$H(\proj^2, \mathcal{L}) > -4$ for all line arrangements
$\mathcal{L}$. But for a general line arrangement or for many
simple examples, the Harbourne constant is at least $-2$. However, there do exist line
arrangements in the plane which have small Harbourne constants. We can
use these to obtain fairly small Harbourne constants for curve
arrangements on ruled surfaces. We illustrate this with two examples
below. 

\begin{example}
Let $X = X_e$ be a rational ruled surface with invariant $e \ge 1$. 
Given a line arrangement in $\proj^2$, one can obtain an arrangement of
curves on $X_e$, following a construction outlined in \cite[Example
15]{E}, where a specific finite morphism $X_e \to X_1$ of degree $e$ is described. Note
that $X_1$ is isomorphic to a blow up of $\proj^2$ at a point. So we
can pull-back lines in $\proj^2$ to $X_e$ which are in the class 
$(1,e)$. If $\mathcal{L}$ is a line arrangement of $d$ lines in the plane, its
pull-back gives a curve arrangement $\mathcal{C}$ of $d$ curves in
$X_e$. 

To be more precise, suppose that $\mathcal{L}$ has $s$ singularities 
and $t_k$ denotes the number of singular points of $\mathcal{L}$ of
multiplicity $k$. Then the singular points of $\mathcal{C}$ are
precisely the pre-images of singularities of $\mathcal{L}$. So 
$\mathcal{C}$  has $es$ singular points and the
number of singular points of multiplicity $k$ is $et_k$. Note that
each curve in $\mathcal{C}$ is in the class $(1,e)$ and has
self-intersection $e$. So the
self-intersection of the divisor associated to $\mathcal{C}$  is
$d^2e$. 

Hence we have $$H(X,\mathcal{C}) = \frac{d^2e-e\sum \limits_{p\in \text{Sing}(\mathcal{L})} r_p^2}{se} =
\frac{d^2-\sum\limits_{p\in \text{Sing}(\mathcal{L})}  r_p^2}{s} =  H(\proj^2, \mathcal{L}).$$

We now assume $e \ge 4$. 
First we consider the Klein arrangement \cite{K}, denoted by $\mathcal{L}_1$. 
This arrangement consists of 21 lines with $t_3 = 28, t_4 = 21$ and
$t_k = 0$ for $k \ne 3,4$. It is easy to see that $H(\proj^2,
\mathcal{L}_1) = -3$. So if $\mathcal{C}_1$ is the curve arrangement in $X$
obtained from $\mathcal{L}_1$, then $H(X,\mathcal{C}_1)=-3$.

Now we calculate the bound given by Proposition \ref{Prop 3.6}. 
(Note that since $ae=b$, this bound is better than the one given by 
 Theorem \ref{Theorem 4.7}.)
We
have $d=21, f_0 = 49e, a=1, b=e, g=0, t_2 = 0, t_3 = 28e, t_4 =
21e$. So Proposition \ref{Prop 3.6} gives 
$$H(X,\mathcal{C}_1) \ge \frac{-9}{2}+
\frac{21}{49e}\left(\frac{4-e}{2}\right)+\frac{21e}{98e}+\frac{9(28)}{8(49)} = \frac{42}{49e}-3.857.$$

Next let $\mathcal{L}_2$  denote the Wiman configuration \cite{W}.  This arrangement consists of 45 lines with $t_3 =
120, t_4 = 45, t_5=36$ and
$t_k = 0$ for $k \ne 3,4, 5$. It is easy to check that
$H(X,\mathcal{C}_2) = -3.359$, where $\mathcal{C}_2$ is the
arrangement of curves in $X$ given by  $\mathcal{L}_2$. 

As above, using Proposition \ref{Prop 3.6}, we obtain
$$H(X,\mathcal{C}_2) \ge \frac{-9}{2}+
\frac{45}{201e}\left(\frac{4-e}{2}\right)+\frac{45e}{402e}+\frac{9(120)}{8(201)} = \frac{90}{201e}-3.828.$$

\end{example}

\section{Ball quotients}\label{ball-quotient}
\textit{Ball quotients} are algebraic surfaces for which the universal cover is
the 2-dimensional unit ball. Equivalently, ball quotients are minimal
smooth complex projective surfaces $Y$ 
of general type satisfying equality in the Bogomolov-Miyaoka-Yau
inequality. In other words, they are minimal
smooth complex projective surfaces $Y$ such that
$K_Y$ is nef and big and $K_Y^2 = 3e(Y),$
where $K_Y$ denotes the canonical divisor and $e(Y)$ is the topological Euler characteristic.
See \cite{T} for more details on ball quotients.

Hirzebruch \cite{H2} gave examples of ball
quotients using line arrangements in $\proj^2$. 
To a line arrangement in $\proj^2$, he associated a surface $Y$ (by
first an abelian cover of $\proj^2$ branched on that line arrangement and
then taking a desingularization). 
He  exhibited three specific line arrangements whose associated surfaces
$Y$ are ball quotients. 

In this section, we show that the surfaces associated to transversal
arrangements on ruled surfaces that we consider in this paper are not ball quotients. 
In order to do this, we use the theory of constantly branched covers developed in
\cite{BHH}. 
The crucial idea is the following. Let $Y$ be a ball quotient 
which arises from the abelian cover construction we used in 
Section \ref{abelian-cover}.  Then if
$E$ is a curve contained in the 
ramification divisor of $\sigma: Y \to \widetilde{X}$, 
then the \textit{relative proportionality}  of $E$ is zero. This is
defined as $\text{prop}(E) := 2E^{2} - e(E)$. 
For more details, see \cite[Section 1.3]{BHH}. See also \cite{H3} for
a nice introduction.  In the notation of \cite{H3}, one says that $Y$
is a \textit{good covering} of $\widetilde{X}$ via $\sigma$.

The same method was used in \cite{P1} and \cite{P2} to study ball quotients.

Let $X$ be a ruled surface with $e \ge 4$. 
Let $\mathcal{C}=\{C_1, C_2, \ldots, C_d\}$ be a transversal arrangement of curves on
the ruled surface $X$ satisfying  Assumption \ref {extra-assumption}.
Let $Y$ be the associated surface constructed in 
Section \ref{abelian-cover}; see Figure \ref{dia:diagram1}. By Theorem \ref{Theorem 4.5}, $K_Y$ is nef and
consequently, $Y$ is a minimal surface of non-negative Kodaira
dimension. In fact, $Y$ is a surface of general type most of the time
as the following remark shows.  

\begin{remark}
Let $\mathcal{C}$ be a transversal arrangement on
the ruled surface $X$ satisfying  Assumption \ref
{extra-assumption}. Assume in addition that $a \ge 8$. 
By \eqref{eq:c_1^2(Y)}, we have 
\begin{align*}
K_Y^2=2^{d-3}\left(32+(8ad-32)g+d(a(2b-ae)+4a(e-2)-8b)+5f_1-9f_0+t_{2}\right).
\end{align*}
Using $a\ge 8$ and Assumption \ref{extra-assumption},  
it is easy to see that  $K_Y^2 >0.$	
Thus $Y$ is a minimal surface of general type.
\end{remark}

We define the Hirzebruch polynomial as
\begin{align*}
H_\mathcal{C}(2):=\frac{1}{2^{d-3}}(3e(Y)- c_1^2(Y)).
\end{align*}
Note that by equation \eqref{3e(Y)-c_1^2(Y)}, we have 
$$H_\mathcal{C}(2)=16-16g+d\left((2b-ae)(5a-2)+4a(g-1)\right)+9f_0-2f_1-4t_2.$$
By Theorem \ref{Theorem 4.5},
$ H_\mathcal{C}(2) \geq 0.$ If $Y$ is a
ball quotient then $H_{\mathcal{C}}(2)=0$.

We now check whether there exists a transversal arrangement
$\mathcal{C}$ on $X$  satisfying  Assumption \ref {extra-assumption}
such that the associated surface $Y$ is a ball quotient. 

As noted above, the relative proportionality of curves 
contained in the ramification divisor of $\sigma$ is zero.
There are two kinds of curves which are contained in
the ramification divisor of $\sigma$.  The first kind are the irreducible
components $F_p$ of $\sigma^{\star}E_p$ for $p \in \text{Sing}(\mathcal{C})$
with $r_p \ge 3$. Since $F_p^2 = -2^{r_p-2}$, \eqref{eq:2-2g} gives 
$\text{prop}(F_p) = 2^{r_{p}-2}(r_{p}-6)$.  

So, if the 
associated surface $Y$ is a ball quotient, then for any point $p \in \text{Sing}(\mathcal{C})$ with $r_p \ge
3$, we have $r_{p}= 6$. Hence the arrangement $\mathcal{C}$
satisfies $t_{k}=0$ for $k\ne 2,6.$


For any $C_i, C_j \in \mathcal{C}$, let $a': = C_i\cdot C_j =
2ab-a^2e$ and $b' := K_X \cdot C_i =
2ae+a(2g-2-e)-2b$. 

For any $j \in \{1,\ldots,d\}$, let $t_k^j$ denote the number 
of $k$-fold points of $C_j$. Since $t_k=0$ for $k \ne 2,6$, 
Lemma \ref{eq:combinatorial equality1}(1) gives
\begin{equation}\label{lin-eq1}
a'(d-1) = 5t_6^j+t_2^j. 
\end{equation}

The second kind of curves contained in
the ramification divisor of $\sigma$ are $D_j := \sigma^{\star}(C_j')$,
where $C_j'$  is the strict transform of 
$C_j$ under the blow up $\tau$. We now calculate the relative 
proportionality $\text{prop}(D_j)$. 


Note that $K_Y = \sigma^{\star}(T)$, where $T$ was defined in
Lemma \ref{Lemma 3.2}. We also recall that, by \eqref{4.2}, we have $T \cdot C_j' =
b'+\frac{a'}{2}+f_0^j-\frac{t_2^j}{2}$. Finally, note that $C_j'^2 = C_j^2 -\sum_{k\ge
  3} t_k^j = a'-\sum_{k\ge   3} t_k^j$.  

Then $\text{prop}(D_j) = 2D_j^2  - e(D_j) = 3D_j^2+K_Y\cdot D_j = 
3\left((\frac{2^{d-1}}{2^2}){C_j'}^2\right)+(\frac{2^{d-1}}{2})(T\cdot
C_j') = 2^{d-3}\left(3a'-3 \sum_{k\ge  3}
  t_k^j\right)+2^{d-3}\left(2b'+a'+2f_0^j-t_2^j\right) =
2^{d-3}\left(4a'+2b'-t_6^j+t_2^j\right)$.

For the final equality above,  we use the fact that $t_k = 0$ for $k
\ne 2, 6$. 
If $Y$ is a ball quotient, then $\text{prop}(D_j) = 0$. This gives
\begin{equation}\label{lin-eq2}
4a'+2b' = t_6^j-t_2^j. 
\end{equation}

Solving the linear equations \eqref{lin-eq1} and \eqref{lin-eq2} for
$t_2^j$ and $t_6^j$, and
using the easy combinatorial identity $\sum_{j=1}^d t_k^j = kt_k$, we
get 
\begin{equation}  \label{constraints}
t_{2} = \frac{a'd^{2}-21a'd-10b'd}{12} \text{ , } t_{6} = \frac{a'd^{2}+3a'd+2b'd}{36}.
\end{equation}

If there exists an arrangement $\mathcal{C}$ on $X$ satisfying
Assumption \ref {extra-assumption} and having only double and sixfold
points such that the associated surface $Y$ is a ball quotient, then 
 $H_\mathcal{C}(2)= 0$. This gives 
 \begin{equation}\label{eq:eq 5.2}
 16-16g+d[(2b-ae)(5a-2)+4a(g-1)]+t_{2}=3t_{6}.
 \end{equation}
  
Plugging the values of $t_2$ and $t_6$ obtained above 
in \eqref{eq:eq 5.2} and simplifying, we get 
  \begin{equation}\label{eq:eq 5.4}
  16-16g=-d\left((3a-1)(2b-ae)+2a(g-1)\right).
  \end{equation}
  We can rewrite \eqref{eq:eq 5.4} as 
  \begin{equation}\label{eq:eq 5.5}
  -16=d[(3a-1)(2b-ae)-2a]+(2ad-16)g.
  \end{equation}
  Thus by our assumptions, we have 
  \begin{align*}
  d[(3a-1)(2b-ae)-2a]&\ge d[(3a-1)ae-2a]=ad[e(3a-1)-2] \\&>0.
  \end{align*}
Note that $d \ge 4$ by Assumption \ref{star}. So if $a\geq2$ or if
$a=1, d\ge 8$, then $(2ad-16)g\geq0$ 
and thus the right-hand side of \eqref{eq:eq 5.5} is a positive number, a contradiction. 
  
Let $a=1$ and $4 \le d \le 7$. Then it is easy to directly check that
\eqref{eq:eq 5.2} is not possible.  First note that the largest value
of $t_6$ is attained when $t_k=0$ for $k\ne 6$ and in this case we
have $t_6 = \frac{a'd(d-1)}{30}$, by Lemma \ref{eq:combinatorial
  equality1}(2). 

If $Y$ is a ball quotient, then  \eqref{eq:eq 5.2}  holds and we have 
\begin{eqnarray*}
0 &=& 16-16g+d[(2b-ae)(5a-2)+4a(g-1)]+t_{2}-3t_{6}\\
& \ge& 16-16g+d(6b-3e+4g-4)-\frac{a'd(d-1)}{10}\\
&\ge& 16-16g+4gd-4d+(2b-e)\left(3d-\frac{d(d-1)}{10}\right)\\
&\ge& 16-4d+4 \left(3d-\frac{d(d-1)}{10}\right), ~~\text{~~~~since~} d\ge
      4, b \ge e \ge 4.
\end{eqnarray*}
Now it is easy to check that the last term above is positive for $4
\le d \le 7$, giving a contradiction.

The above arguments prove the following theorem. 
  \begin{theorem}\label{Theorem 5.2}
Let $X$ be a ruled surface with $e \ge 4$. 
There does not exist any transversal arrangement $\mathcal{C}$
on $X$ satisfying  Assumption \ref{extra-assumption} 
such that the associated surface $Y$ is a ball quotient.
  \end{theorem}

{\bf Acknowledgement:} We thank Piotr Pokora for reading
  this paper and giving many useful suggestions. We also thank the
  referees for making several helpful suggestions which substantially
  improved the paper.

\bibliographystyle{plain}

\end{document}